\documentclass{amsart}
\usepackage{amsfonts}
\usepackage{color}

\usepackage{graphicx}

\newtheorem{thm}{Theorem}[section]
\newtheorem{cor}[thm]{Corollary}
\newtheorem{lema}[thm]{Lemma}
\newtheorem{prop}[thm]{Proposition}
\theoremstyle{definition}
\newtheorem{defn}[thm]{Definition}
\theoremstyle{remark}
\newtheorem{rem}[thm]{Remark}
\theoremstyle{example}

\numberwithin{equation}{section}
\newcommand{\R}{\mathbb R}
\newcommand{\N}{\mathbb N}
\newcommand{\MM}{\mathcal{M}}
\newcommand{\LL}{\mathcal{L}}

\newcommand{\KK}{\mathcal{K}}

\newcommand{\ve}{\varepsilon}
\newcommand{\lam}{\lambda}

\begin{document}
\title[\hfil PRINCIPAL EIGENVALUES]
{Principal eigenvalues of fully nonlinear integro-differential elliptic equations with a drift term}
\author{Alexander Quaas, Ariel Salort and Aliang Xia}

\address{Alexander Quaas\hfill\break\indent
Departamento de Matem\'{a}tica, Universidad T\'ecnica Federico Santa Mar\'{\i}a
\hfill\break\indent  Casilla V-110, Avda. Espa\~na, 1680 --
Valpara\'{\i}so, CHILE.}
\email{{\tt alexander.quaas@usm.cl}}

\address{Ariel Salort \hfill\break\indent
CONICET  \hfill\break\indent
Departamento de Matem\'{a}tica, FCEyN UBA
\hfill\break\indent Ciudad Universitaria, Pab I (1428)
\hfill\break\indent Buenos Aires,
ARGENTINA. }
\email{{\tt asalort@dm.uba.ar}}

 \address{Aliang Xia \hfill\break\indent
 Department of Mathematics, Jiangxi Normal University
\hfill\break\indent Nanchang, Jiangxi 330022,
P.~R.~China.}
\email{{\tt xiaaliang@126.com}}

\subjclass[2010]{35J60; 47G20; 35P30}

\keywords{Principal eigenvalue, integro-differential equation, regularity, Krein-Rutman theorem.}

\begin{abstract}
 We study existence of  principal eigenvalues of  fully nonlinear integro-differential elliptic equations with a drift term via the Krein-Rutman theorem which based on  regularity up to boundary of viscosity solutions. We also show the simplicity of the eigenfunctions in viscosity sense by a nonlocal version of ABP estimate.
\end{abstract}

\maketitle

%**********************************************************************************************
%********************************

%\begin{abstract}
 
%\end{abstract}

 \section{Introduction}
 In this article,  we study the regularity of viscosity solutions and spectral properties of non-divergence integro-differential equations. To be more precise, we consider non-local elliptic equations with a drift term with the following form
\begin{equation} \label{opp}
	 I u(x)= \inf_{a\in A}\sup_{b \in B}\left\{L_{K_{a,b}}u(x) + c_{a,b}(x)\cdot \nabla u(x)\right\}=0,
\end{equation}
where $\{L_{K_{a,b}}\}_{a\in A, b\in B}$ a family of integro-differential operators defined by 
\begin{align*}
	&L_{K_{a,b}}u=\int_{\R^n} \delta(u,x,y) K_{a,b}(y) \, dy,\\
	&\delta(u,x,y)=u(x+y)+u(x-y)-2u(x).
\end{align*}
The function  $c_{a,b}$ is assumed to be uniformly bounded in $\Omega$ and the family of kernels $\{K_{a,b}\}_{a\in A, b\in B}$ are symmetric and comparable with  the respective kernel of the fractional laplacian $-(-\Delta)^s$, for $s\in(0,1)$. 

Equations of type \eqref{opp} arise from stochastic control problems,  namely in competitive stochastic games with two or more players,  which are  allowed to choose from
different strategies at every step in order to maximize the expected value of some functions at the first exit point of a domain, see for instance \cite{H}. The integro-differential equation like (\ref{opp}) correspond to purely jump processes when diffusion and drift are neglected, which have been studied intensively in the last years, see \cite{CS,CS2,ROS,S} and references therein and Chang-Lara \cite{CL} considered the case with a drift term and the kernel is uniformly and not-symmetric. 

In this article, we also consider the operator  $I$ as in \cite{CL}. More precise, we are interested in studying  the equation $-Iu=f$ in a given domain $\Omega$, $u$ being a function vanishing outside the domain, and $f$ is assumed to be a continuous function. This problem, and a generalization to possibly non-symmetric kernels, was treated in \cite{CL}, where existence of solutions and interior regularity results were obtained by using the same techniques that in \cite{BCI}, \cite{CC}, \cite{S} and \cite{KS}. 
 In our paper, we discuss  $C^\alpha$ regularity up to the boundary by using the ideas in \cite{BDGQ} to analyze the behavior of the maximal Pucci operator near the boundary . Then, having those results, we are aimed at establishing the existence of the principal eigenvalues corresponding to operator $-I$ with Dirichlet boundary conditions via the classical Krein-Rutman theorem  \cite{KR,M} and compactness arguments. 

The eigenvalue problems have been  extensively studied for  nonlinear  operators, we give a quick review of its here. 
In \cite{P}, Pucci first noticed the phenomena of nonlinear operators possessing two principle half-eigenvalue (or semi-eigenvalue, or demi-eigenvalue). It also discovered by Berestycki \cite{B} for Sturm-Liouville equations. An important step in studying these types of questions was made by Lions \cite{L}, who used stochastic methods to study the principle half-eigenvalues of certain Bellman operators and also the ideas of Berestycki, Nirenberg and Varadhan \cite{BNV}, who discovered deep connections between the maximum principle  and principle eigenvalues of linear operators. The question of existence of principle eigenvalues of Pucci extremal operator studied by Felmer and Quaas \cite{FQ2}. The principal eigenvalues for fully nonlinear uniformly elliptic operators in non-divergence form as well as homogeneous and convex  (or concave) was considered by  Quaas and Sirakov \cite{QS} and \cite{QS2}. Ishii and Yoshimura \cite{IY} and  Armstrong \cite{A} showed analogous results as \cite{QS} for operators which not necessarily convex, such as Bellman-Isaacs operator. Birindelli and Demengel \cite{BD,BD2} have show similar results for certain nonlinear operators which are degenerate elliptic.
For more on principle eigenvalues of nonlinear elliptic operators, we refer reader to \cite{BEQ,S2} and references therein. In this article, we focus on the principle half-eigenvalues of 
non-local fully nonlinear elliptic operator $-I$. 

We make the convention that any time we say a non-regular function satisfies an (in)equality, we shall mean it is satisfies in the viscosity sense-see for example \cite{CL,CS2} for definitions and properties of these.

With this in mind, following the definitions in \cite{BNV,QS}, we define the following (finite, see Lemma \ref{p1} blow) quantities
\begin{eqnarray*}
\lam_1^+(I,\Omega)& = \sup\{\lam\, : \, \exists v\in C(\bar{\Omega})\cap L^1(\omega_s), v>0 \mbox{ in }\Omega  \mbox{ and } v\ge0 \mbox{ in }\R^n\setminus\Omega \\
&\mbox{ such that }  Iv+\lam v\leq 0 \mbox { in } \Omega \}, \\
\lam_1^-(I,\Omega) &= \sup\{\lam\, : \, \exists v\in C(\bar{\Omega})\cap L^1(\omega_s), v<0 \mbox{ in }\Omega \mbox{ and } v\le0 \mbox{ in }\R^n\setminus\Omega\\
& \mbox{ such that }  Iv+\lam v\geq 0 \mbox { in } \Omega \},
\end{eqnarray*}
where weight function $\omega_s$ is given in section 2 such that the operator is well-defined. Then $\lam_1^+(I,\Omega)$  and $\lam_1^-(I,\Omega)$ are the principal half-eigenvalues of $-I$ in $\Omega$.

Now, we can state our main results. Our first result is
\begin{thm} \label{teo.main}
Let $\Omega$ be a $C^2$ bounded domain of $\R^n$  and assume $s\in (\frac{1}{2},1)$. There exists functions $\phi^+, \phi^-\in C(\bar{\Omega})\cap L^1(\omega_s)$ such that $\phi^+>0$ and $\phi^-<0$ in $\Omega$, and which satisfy
 \begin{equation*} 
\begin{cases}
	-I\phi^+=\lam^+_1(I,\Omega) \phi^+&\qquad {\mbox in }\,\,\Omega,\\
	-I\phi^-=\lam^-_1(I,\Omega) \phi^-&\qquad{ \mbox in }\,\,\Omega,\\
	\phi^+=\phi^-=0  &\qquad {\mbox in }\,\,\R^n\setminus \Omega.
\end{cases}	
\end{equation*} 
\end{thm}
From here, we say eigenvalue $\lam^+_1(I,\Omega)$ (resp. $\lam^-_1(I,\Omega)$) corresponds a eigenfunction $\phi^+>0$ (resp. $\phi^-<0$).

Next, we use the Aleksandrov-Bakelman-Pucci (ABP) estimate and some techniques in \cite{BNV} (see also \cite{QS} and \cite{A}) to prove the simplicity of eigenfunctions. Therefore, we can get our second result.
\begin{thm} \label{teo.main 2}
Let $\Omega$ be a $C^2$ bounded domain of $\R^n$  and assume  $s\in (\frac{1}{2},1)$. Assume there exists a viscosity  solution $u\in C(\bar{\Omega})\cap L^1(\omega_s)$ of
 \begin{equation} \label{1.2}
\begin{cases}
	-Iu=\lam^+_1(I,\Omega)u&\qquad {\mbox in }\,\,\Omega,\\
	u=0  &\qquad {\mbox in }\,\,\R^n\setminus \Omega,
\end{cases}	
\end{equation} 
or of 
 \begin{equation} \label{1.3}
\begin{cases}
	-Iu\leq \lam^+_1(I,\Omega)u&\qquad {\mbox in }\,\,\Omega,\\
	u(x_0)>0, \quad u\leq0  &\qquad {\mbox in }\,\,\R^n\setminus \Omega,
\end{cases}	
\end{equation} 
for some  $x_0\in\Omega$. Then $u=t\phi^+$ for some $t\in\R$. If a function $v\in C(\bar{\Omega})\cap L^1(\omega_s)$ satisfies either (\ref{1.2}) or the reverse inequalities in (\ref{1.3}),
with $\lam^+_1(I,\Omega)$ replaced $\lam^-_1(I,\Omega)$, then $v=t\phi^-$ for some $t\in\R$.
\end{thm}

The main tool to obtain the two principle half-eigenvalues is  the classical Krein-Rutmann theorem \cite{KR} and compactness arguments which based on the regularity up to boundary. The regularity up to boundary of solutions involving integro-differential operators was considered by Ros-ton and Serra \cite{ROS}. In our paper, we discuss  regularity up to boundary of viscosity solutions to  integro-differential operators with a gradient term. We should remark that we just consider the case $s\in (1/2,1)$ in our article since we need to more regularity to ensure the operator $I$ is well-defined, see also \cite{CL}. In fact, if coefficient $c_{a,b}\equiv0$ in $\Omega$ (cf. fractional Laplacian operator), we can prove our Theorems \ref{teo.main} and \ref{teo.main 2} are still true for all range $s\in(0,1)$.

We remark that the principal eigenvalue is  a starting point to study Rabinowitz bifurcation-type results, solutions at resonance, Ladezman-Lazer type results and Ambrosetti- Prodi phenomenon, see for example \cite{R,A1,FQS1,FQS2,S2} and references therein.

This article is organized as follows. In Section 2, we recall some definitions and some useful and known results. The regularity up to the boundary for Dirichlet problem involving operator (\ref{opp}) is obtained in Section 3. Section 4 is devoted to prove the ABP estimate related the operator (\ref{opp}). We prove  a technical lemma (that is, (H) condition) in order to apply the Krein-Rutman theorem in Section 5. We prove our main theorems, Theorems \ref{teo.main} and \ref{teo.main 2}, in Section 6. Finally, in section 7 we make an application of the principle eigenvalues.

\section{Preliminars}
To be precise about the formulas we presented in the introduction, we need to ask an integrability condition for the kernels around the origin. Throutout the paper we denote $\LL$ the class of all the linear operators given in \eqref{opp}, and, given $L\in\LL$ we assume that the operator $Lu(x)$ is defined for $u\in C^{1,1}(x) \cap L^1(\omega_s)$, where 
$$\omega_s(dy)=\min\{1,|y|^{-(n+2s)}\}dy.$$

We remark that the family of extremal Pucci operator for a function $u$ are computed at a point $x$ by
$$
	\mathcal{M}_{\mathcal{L}}^+ u(x)=\sup_{L\in \mathcal{L}} Lu(x), \qquad
	\mathcal{M}_{\mathcal{L}}^- u(x)=\inf_{L\in \mathcal{L}} Lu(x).
$$
Observe that $\LL$ and $\mathcal{M}^\pm_\LL$ depend on some additional parameters $\lam$, $\Lambda$ and $s$, but we do not make it explicit to do not overcharge the notation. 

We also say that an operator $I$ defined over a domain $\Omega\subset\R^n$ is elliptic with respect to the family of linear operators $\LL$ if for every $x\in \Omega$ and any pair of functions $u$ and $v$ where $Iu(x)$ and $Iv(x)$ can be evaluated, then also $Lu(x)$ and $Lv(x)$ are well defined and 
$$
	\mathcal{M}_{\mathcal{L}}^-(u-v)(x)\leq I u(x)-Iv(x) \leq \mathcal{M}_{\mathcal{L}}^+(u-v)(x).
$$

   $Lu$ is continuous in $B_r(x_0)$ if $u\in C^2(B_r(x_0) \cap L^1(\omega_s)$. Stability properties of $I$ depend on $Iu$ being continuous when $u$ is sufficient regular, in this case, $C^2(B_r(x_0)\cap L^1(\omega_s))$ is a reasonable requirement. As in \cite{CS2}, we define continuous elliptic operators as follows.

\begin{defn}
we say that $I$ is a \emph{continuous operator, elliptic with respect to} $\LL=\LL(\KK)$ in $\Omega$ if,
\begin{itemize}
	\item[(1)] $I$ is an elliptic operator with respect to $\LL$ in $\Omega$, 
	\item[(2)]	$Iu(x)$ is well defined for any $u\in C^{1,1}(x) \cap L^1(\omega_s)$ and $x\in \Omega$, 
	\item[(3)] $Iu$ is continuous in $B_r(x_0)$ for any $u\in C^2(B_r(x_0))\cap L^1(\omega_s)$ and $B_r(x_0)\subset \Omega$.
\end{itemize}	
\end{defn}

In the hypothesis we introduce we see that the non-local term of the family $\LL$ is actually obtained bounding our kernels by multiples of the kernel of the fractional laplacian. Along the paper, unless it is stated otherwise, it is assumed that $s\in (\frac{1}{2},1)$.

\subsection{Hypothesis}
We assume the following hypothesis on the family $\LL$ depending on a family of kernels $\KK$ and some additional parameters in the following way:
\begin{itemize}
	\item[(H1)] Every $I\in\mathcal{L}$ is of the form $I=L_K+c_{a,b}\cdot \nabla$ for $K\in\mathcal{K}$.
	
	\item[(H2)] There are constants $\lam\leq \Lambda$, such that for every $K\in\mathcal{K}$, 
	$$
		\frac{\lambda}{|y|^{n+2s}} \leq K(y) \leq \frac{\Lambda}{|y|^{n+2s}} 
	$$
	for $\Lambda\ge\lambda>0$.
	
	\item[(H3)] Operator $I$ is positive 1-homogenous in $u$, that is, $I(tu)=tI(u)$, for $t\ge0$.  
	
	\item[(H4)]   There is a constant $c^+>0$ such that
$$
	|c_{a,b}|\leq c^+ \mbox{uniformly in } \Omega.
$$
%	 $$
%	 	r^{2s-1}\left| c+\int_{B_1\setminus B_r} y K(y) \, dy \right|\leq c^+ \qquad \mbox{ for every } r\in(0,1)
%	 $$
\end{itemize}
In this settings we can also write
$$
	\mathcal M_{\mathcal L}^\pm u(x)= \mathcal M_{\mathcal K}^\pm u(x) \pm c^+ |Du(x)|,
$$
with
\begin{align} \label{mk}
	\mathcal M_{\mathcal K}^+ u(x) =
		\sup_{K\in \mathcal K} L_K u(x) =  \int_{\R^N} \frac{S^+(\delta (u,x,y))}{|y|^{n+2s}} \, dy, \\
	\mathcal M_{\mathcal K}^- u(x) =
		\sup_{K\in \mathcal K} L_K u(x) =  \int_{\R^N} \frac{S^-(\delta (u,x,y))}{|y|^{n+2s}} \, dy,
\end{align}
where given $t\in \R$ we denote
	$$
		S^+(t)=\Lambda t_+ - \lambda t_-, \qquad 
		S^-(t)=\lambda t_+ - \Lambda t_-.
	$$

A particular example of operator that satisfies all the previous hypothesis are
$$
		\mathcal{L}=\{L_c=-(-\Delta)^{s}+c\cdot \nabla \, : \, |c|\leq c^+ \}
$$ 
where  the fractional Laplacian is defined as
$$
	-(-\Delta)^s u(x)=\int _{\R^n}\frac{\delta(u,x;y)}{|y|^{n+2s}} \, dy.
$$
We notice that the operator $\LL$ that we have defined belongs to  the more general class treated in \cite{CL}, where  no symmetry assumption on the kernels is made. 

\medskip

Let us fix some notations we will use along the paper. From now on we define for $\delta>0$ the set
$$
	\Omega_\delta:=\{y\in \Omega \,:\, d(y)<\delta\}.
$$
Also, along this paper we denote $d(x)$ the distance of $x$ to $\partial \Omega$, that is, 
$$
	d(x):=dist(x,\partial \Omega),\qquad x\in \Omega.
$$
It is well known that $d$ is Lipschitz continuous in $\Omega$ with Lipschitz constant $1$ and it s a $C^2$ function in a neighborhood of $\partial\Omega$ (see \cite{GT}, p. 354). We modify it outside this neighborhood to make it a $C^2$ function (still with Lipschitz constant 1), and we extend it to be zero outsider $\Omega$.

 Then we define our barrier function as follows
\begin{equation} \label{f.xi}
	\xi(x)=
	\begin{cases}
		d(x)^\beta &\mbox{if }x\in \Omega_\delta,\\
		\ell &\mbox{if }x \in \Omega\setminus \Omega_\delta,\\
		0 &\mbox{if }x \in \R^n\setminus\Omega
			\end{cases}
\end{equation}
for $\beta>0$ and a function $\ell$ such that $\xi$ is positive and $C^2$ in $\Omega$.

\subsection{Preliminary results}

In this section we present some results concerning   the family $\LL$. We denote the set of upper (resp. lower) semicontinuous in $\Omega$ by $USC(\Omega)$ (resp. $LSC(\Omega)$).
Then, we define the notion of viscosity solution in this setting as \cite{CL} (see also \cite{CS2}).

\begin{defn}
		Given a non local operator $I$ and a function $f:\Omega\to \R$ we say that $u\in LSC(\Omega)\cap L^1(\omega_s)$ is a \emph{super-solution (sub-solution)} to 
		$$
			Iu\leq (\geq) f \quad \mbox{ in the viscosity sense in } \Omega,
		$$
		if for every point $x_0\in\Omega$ and any neighborhood $V$ of $x_0$ with $\bar V\subset \Omega$ and for any $\varphi\in C^2(\bar V)$ such that $u(x_0)=\varphi(x_0)$ and
	$$
		u(x)<\varphi(x) \quad(\mbox{resp. } u(x)>\varphi(x)) \quad \mbox{ for all }x\in V\setminus \{x_0\}
	$$
	the function $v$ defined by
	$$
		v(x)=u(x)\quad \mbox{ if }x\in \R^n\setminus V \quad \mbox{ and }\quad v(x)=\varphi(x)\quad \mbox{ if }x\in V
	$$
	satisfies 
	$$
		Iv(x_0)\leq  f(x_0) \quad (\mbox{resp. }-Iv(x_0)\geq f(x_0).
	$$	
	Additionally, $u\in C(\Omega)\cap L^1(\omega_s)$ is a viscosity solution to $Iu=f$ in $\Omega$ if it is simultaneously a sub-solution and a super-solution.
\end{defn}

\begin{rem}
	(1) As in the usual definition, we may consider inequality instead strict inequality 
	$$
		u(x)\geq \varphi(x)\quad \mbox{ for all }x\in V\setminus{x_0},
	$$
	and "in some neighborhood $V$ of $x_0$" instead "in all neighborhood".
	
	(2)Other definitions and their equivalence can be founded in \cite{BCI}.
\end{rem}

A useful tool to be used is the following comparison principle between sub and super-solution   proved in \cite{CL}, Corollary 2.9.
\begin{lema} [Comparison principle]  \label{lema.comp}
Let $u\in LSC(\Omega)\cap L^1(\omega_s)$ and $v\in USC(\Omega)\cap L^1(\omega_s)$ be a super-solution and a sub-solution, respectively, of the same equation $Iw=f$ in $\Omega$ . Then $u\geq v$ in $\R^n \setminus \Omega$ implies $u\geq v$ in $\R^n$.
\end{lema}

Also, a result related to the difference of solutions is proved in in \cite{CL} .

\begin{thm} \label{teo.sol}
Let $I$ be a uniformly elliptic operator with respect to $\LL$, and $f$, $g$ continuous functions. Given $u\in LSC(\Omega)\cap L^1(\omega_s)$ and $v\in USC(\Omega)\cap L^1(\omega_s)$ such that $Iu\leq f$ and $Iv\geq g$ hold in $\Omega$ in the viscosity sense, then $\MM_\LL^-(u-v)\leq f-g$ also holds in $\Omega$ in the viscosity sense.
\end{thm}

By using the Perron's method together with the comparison principle it follows the existence and uniqueness of solution for the operator $I$ in the viscosity sense, \cite{CL}.

\begin{thm} \label{t2.6} Given a domain $\Omega\subset \R^n$ with the exterior ball condition, a continuous operator $I$ with respect to $\LL$ and $f$ and $g$ bounded and continuous functions (in fact $g$ only need to be assumed continuous at $\partial \Omega$), then the Dirichlet problem
\begin{equation} \label{main.eq}
\begin{cases}
	Iu=f&\qquad \mbox{ in }\Omega,\\
	u=g  &\qquad \mbox{ in }\R^n\setminus \Omega,
\end{cases}	
\end{equation}
has a unique viscosity solution $u$.

\end{thm}

A stability result for sub-solutions is stated in \cite{CL}. Naturally, a corresponding result it holds for super-solutions by changing $u$ to $-u$. As a corollary, the stability under uniform limits follows.
\begin{prop} \label{prop.cont}
Let $\{f_k\}$ be a sequence of continuous functions and $\{I_k\}$ a sequence of elliptic operators with respect to $\LL$. Let $u_k\in LSC(\Omega)\cap L^1(\omega_s)$ be a sequence of   functions in $\R^n$  such that
\begin{enumerate}
\item[(a)] $I u_k =f_k \mbox{ in } \Omega$,
\item[(b)] $u_k \to u \mbox{ locally uniformly in}\,\, \Omega$,
\item[(c)] $u_k \to u \mbox{ in } L^1(\omega_s)$,
\item[(d)] $f_k \to f \mbox{ locally uniformly in } \Omega$,
\item[(e)] $|u_k(x)|\leq C \mbox{ for every }x\in \Omega$.
\end{enumerate}
Then $Iu=f$ in the viscosity sense in $\Omega$.
\end{prop}

Another useful result to be established is the following version of the strong maximum principle.
 
\begin{thm} [Strong Maximum Principle] \label{teo.pm}
Let $u\in LSC(\Omega)\cap L^1(\omega_s)$ be a viscosity super-solution of $-\mathcal{M}_\LL^- u\geq  0$, $u\ge 0$ in $\R^n$. Then either $u>0$ in $\Omega$ or $u\equiv 0$ in $\Omega$.
\end{thm}
\begin{proof}
The proof is very similar as Lemma 7 in \cite{BDGQ}, and we omit it here.
\end{proof}

\medskip

\section{Regularity}
In this section we prove regularity up to the boundary for the equation
\begin{equation} \label{main.eq.1}
\begin{cases}
	-Iu=f &\qquad \mbox{ in }\Omega\\
	u=0  &\qquad \mbox{ in }\R^n\setminus \Omega.
\end{cases}	
\end{equation}
As usual, if for a fixed $\delta>0$ small enough we denote $\Omega_\delta$ a $\delta-$neighborhood of $\Omega$,  the global regularity follows by studying  the regularity both in $\Omega \setminus \Omega_\delta$ and $\Omega_\delta$.
Nevertheless, before dealing with the regularity we prove lower and upper bounds of the  extremal Pucci operators defined in \eqref{mk} for powers of the distance to the boundary. In order to state such result, we remember the following result proved in Proposition 2.7, \cite{ROS}. Given $\beta\in (0,2s)$, we  denote   $\varphi^\beta:\R \to \R$ the function 
\begin{equation} \label{funcfi}
	\varphi^\beta(x):=(x_+)^\beta.
\end{equation}

\begin{lema} \label{lema.beta}
Given $s\in (0,1)$, for $\beta\in (0,2s)$ the function \eqref{funcfi} satisfies
	\begin{align*}
	\MM^+_\KK (\varphi^\beta) = c^+(\beta) x^{\beta-2s} \qquad \mbox{and}  \qquad 	\MM^-_\KK (\varphi^\beta) = c^-(\beta) x^{\beta-2s} \qquad \mbox{ in }\{x>0\}.
	\end{align*}
	The constants $c^+$ and $c^-$ depend on $s$, $\beta$ and $n$, and are continuous as functions of the variables $s$ and $\beta$ in $\{0<s\leq 1,\,  0<\beta<2s\}$. Moreover, there are $\beta_1\leq \beta_2$ in $(0,2s)$ such that
	$$
		c^+(\beta_1)=0 \quad \mbox{ and } \quad c^-(\beta_2)=0.
	$$
	Furthermore, 
	$$
		c^+(\beta)<0  \mbox{ if } (0,\beta_1), \qquad c^+(\beta)>0  \mbox{ if } (\beta_1,2s)
	$$
	$$
		c^-(\beta)<0  \mbox{ if } (0,\beta_2), \qquad c^-(\beta)>0  \mbox{ if } (\beta_2,2s).
	$$
	In particular, for the fractional Laplacian $-(-\Delta)^s$ it holds that $\beta_1\leq s \leq \beta_2$.
\end{lema}

We state the behavior of the extremal operators regarding the barrier function $\xi$ defined in \eqref{f.xi}.

First we prove the following technical lemma.

\begin{lema} \label{lema.c}
 
Let $\{x_k\}_{k\in\N}\in \Omega$ be a sequence such that $d(x_k)\to 0$ as $k\to\infty$ and let us denote $d_k:=d(x_k)$. Given $\beta\in (0,2)$ let the function
	$$
	g_k(z):=\left(\frac{d(x_k+d_k z)}{d_k}\right)^{\beta}+
			\left(\frac{d(x_k-d_k z)}{d_k}\right)^{\beta}-2.
	$$
	Then
	\begin{equation*}
			\lim_{k\to \infty}\int_{\R^n} \frac{  S^+ (g_k(z)) }{ |z|^{n+2s} }  dz  = c^+(\beta),  \qquad 
			\lim_{k\to \infty}\int_{\R^n} \frac{  S^- (g_k(z)) }{ |z|^{n+2s} }  dz  = c^-(\beta)	,		
	\end{equation*}	
	where $c^\pm(\beta)$ are given in Lemma \ref{lema.beta}.
 
\end{lema}
\begin{proof}
 
	Given $L$ and $\eta$ fixed positive values, we split the  integral involving $S^+$ as follows
	\begin{align*}
		\int_{\R^n} \frac{  S^+ (g_k(z)) }{ |z|^{n+2s} } dz&=
		\int_{|z|\geq L} \frac{S^+(g_k(z))}{|z|^{n+2s}}dz + \int_{|z|\leq \eta }\frac{S^+(g_k(z))}{|z|^{n+2s}}dz + \int_{\eta\leq |z|\leq L}\frac{S^+(g_k(z))}{|z|^{n+2s}}dz\\
		&:= I_1 + I_2+ I_3.
	\end{align*}
	Observe that
	\begin{align*}
			|I_1| &\leq  \int_{|z|\geq L} \frac{|S^+(g_k(z))|}{|z|^{n+2s}}\,dz \leq \Lambda 
		\int_{|z|\geq L} \frac{g_k^+(z)}{|z|^{n+2s}} dz + \lam \int_{|z|\geq L} \frac{g_k^-(z)}{|z|^{n+2s}} dz \\
		& \leq \Lambda \int_{|z|\geq L} \frac{|g_k(z)|}{|z|^{n+2s}}  dz :=\Lambda |I_1'|.
	\end{align*}
	Let us deal with $|I'_1|$. Observe that when $x_n+d_n  z\in \Omega$, we have by the Lipschitz property of $d$ that $d(x_k+d_k z) \le d_k (1+|z|)$.
	Of course, the same is true when $x_k +d_k z \not\in \Omega$ and it similarly follows that $d(x_k-d_k z)
	\le d_k (1+|z|)$. Thus, we obtain for large $k$
	\begin{equation}\label{ineq1}
		\displaystyle |I_1'|\le \displaystyle 2 \int_{ |z| \ge L} \frac{1+(1+|z|)^{\beta}}{|z|^{n+2s}} dz.
	\end{equation}
	Observe that the previous expression tends to zero as $L\to+\infty$.

	A similar computation leads to
	$$
		|I_2| \leq  \Lambda \int_{|z|\leq \eta} \frac{|g_k(z)|}{|z|^{n+2s}}  dz :=\Lambda |I_2'|.
	$$

	Let us deal with $|I_2'|$. Since $d$ is smooth in a neighborhood of the boundary, when $|z|\le L$
	and $x_k +d_k z\in \Omega$, we obtain by Taylor's theorem
	\begin{equation}\label{taylor}
		d(x_k + d_k z )= d_k + d_k \nabla d(x_k) z + \Theta_n(d_k,z) d_k^2 |z|^2,
	\end{equation}
	where $\Theta_k$ is uniformly bounded, i.e, $-C\leq \Theta_k \leq C$ for some positive constant $C$. Hence
	\begin{equation}\label{eq3}
		 |d(x_k+d_k z ) - (d_k + d_k \nabla d(x_k) z)| \le   C d_k^2 |z|^2.
	\end{equation}
	Now choose $\eta \in (0,1)$ small enough. Since $d(x_k)\to 0$ and $|\nabla d|=1$ in a
	neighborhood of the boundary, we can assume that
	\begin{equation}\label{extra1}
	\nabla d(x_k)\to e \hbox{ as }k\to +\infty \hbox{ for some unit vector }e.
	\end{equation}
	Without loss of generality, we may take
	$e=e_n$, the last vector of the canonical basis of $\R^n$. If we restrict $z$ further to satisfy $|z|\le \eta$,
	we obtain $1+\nabla d(x_k) z \sim 1 + z_n \ge 1-\eta>0$ for large $k$, since $|z_k| \le |z|\le \eta$.
	Therefore,   inequality \eqref{eq3} is also true when
	$x_k+d_k z\not\in \Omega$ for large $k$  (depending only on	$\eta$). Moreover, by using again Taylor's theorem
	$$
		|(1+\nabla d(x_k) z \pm  C d_k |z|^2)^{\beta} - (1+ \beta \nabla d(x_k) z)|  \leq   C  |z|^2,
	$$
	for large enough $k$. Thus from \eqref{eq3},
	$$
		\left|\left(\frac{d(x_k+d_k z)}{d_k}\right)^{\beta} - (1+ \beta \nabla d(x_k) z)\right| \leq  C |z|^2,
	$$
	for large enough $k$. A similar inequality is obtained for the term involving $d(x_k-d_k z)$, i.e., 
	$$
			\left|\left(\frac{d(x_k-d_k z)}{d_k}\right)^{\beta} - (1- \beta \nabla d(x_k) z)\right| \leq  C |z|^2,
	$$
	for large enough $k$, and consequently we deduce that
	$|g_k(z)|\leq C |z|^2$ for $k$ large enough. Therefore	
	$|I_2'|$ can be bounded as
	\begin{equation}\label{ineq2}
	\begin{array}{l}
		|I_2'| \leq  C  \int_{ |z| \le \eta} |z|^{2(1-s)-n} dz.
	\end{array}
	\end{equation}
	Observe that the previous expression tends to zero as $\eta\to 0$.

	We finally observe that it follows from the above discussion (more precisely from \eqref{taylor} and
	\eqref{extra1} with $e=e_n$) that for $\eta \le |z| \le L$
	\begin{equation}\label{extra2}
	\frac{d(x_k \pm d_k z)}{d_k} \to (1\pm z_n)_+ \qquad \hbox{as } k \to +\infty 
	\end{equation}
	and, as $k\to\infty$,  by dominated convergence we arrive at
	\begin{align} \label{eclim}
	\begin{split}
	 	\int_{\eta\leq |z|\leq L}\frac{S^+(g_k(z))}{|z|^{n+2s}}
		 dz =\int_{ \eta \le |z| \le L}  \frac{S^+( (1+z_n)_+^{\beta}+ (1-z_n)_+^{\beta}-2)}{|z|^{n+2s}} dz.
	\end{split}
	\end{align}
	In consequence, from \eqref{ineq1} and \eqref{ineq2} it follows that, as $k\to \infty$, the difference
	$$
		\left|\int_{\R^n} \frac{  S^+ (g_k(z)) }{ |z|^{n+2s} }  dz  - \int_{ \eta \le |z| \le L}  \frac{S^+( (1+z_n)_+^{\beta}+ (1-z_n)_+^{\beta}-2)}{|z|^{n+2s}} dz\right|
	$$
	can be bounded, up to a multiplicative constant independent on $L$ and $\eta$,  by
	$$
		 \int_{ |z| \ge L} \frac{1+(1+|z|)^{\beta}}{|z|^{n+2s}} dz +     \int_{ |z| \le \eta} |z|^{2(1-s)-n} dz,
	$$
	from where, as $\eta\to 0$ and $L\to \infty$, we obtain that
	\begin{equation} \label{integ.1}
			\lim_{k\to\infty}\int_{\R^n} \frac{  S^+ (g_k(z)) }{ |z|^{n+2s} }  dz  = \int_{ \R^n}  \frac{S^+( (1+z_n)_+^{\beta}+ (1-z_n)_+^{\beta}-2)}{|z|^{n+2s}} dz.
	\end{equation}

	It is well-known, with the use of Fubini's theorem and a change of variables, that the integral in the right side of \eqref{integ.1} can be rewritten as a one-dimensional integral 
	\begin{equation}\label{contra-final2}
	c^+(\beta):=\int_ \R \frac{S^+((1+t)_+^{\beta}+ (1-t)_+^{\beta}-2)}{|t|^{1+2s}} dt,
	\end{equation}
	from where the result follows.
 	
\end{proof}

The behavior of the extremal operators regarding the barrier function $\xi$ is established in the following result.

\begin{lema}  \label{lema.m}
Let $\Omega$ be a $C^2$ bounded domain in  $\R^n$, $s\in (0,1)$ and $\xi$ the function defined in \eqref{f.xi}. There exist $C,\delta>0$ such that
\begin{itemize}
\item[(a)] $\mathcal{M}^+_\KK (\xi(x)) \geq C  d^{\beta-2s}(x) \qquad \mbox{\, \,  if } \beta\in (\beta_1,2s)$,

\item[(b)] $\mathcal{M}^+_\KK (\xi(x)) \leq -C  d^{\beta-2s}(x) \qquad \mbox{ if } \beta\in (0,\beta_1)$,

\item[(c)] $\mathcal{M}^-_\KK (\xi(x)) \geq C  d^{\beta-2s}(x) \qquad \mbox{\, \, if } \beta\in (\beta_2,2s)$,

\item[(d)] $\mathcal{M}^-_\KK (\xi(x)) \leq -C  d^{\beta-2s}(x) \qquad \mbox{ if } \beta\in (0,\beta_2)$
\end{itemize}
for $x\in\Omega_\delta$, where $0<\beta_1<\beta_2<2s$ are given in Lemma \ref{lema.beta}.

\end{lema}

\begin{proof}
 
	Let us prove (a).	By contradiction, let us assume that the conclusion of the lemma is not true. Then there exist $\beta\in
	(\beta_1,2s)$ and sequences of points $x_k\in \Omega$ with $d(x_k)\to 0$  such that
	\begin{equation}\label{contradiction}
		\lim_{k\to +\infty} d(x_k)^{2s-\beta} \MM^+_\KK( d^{\beta}(x_k)) \le 0.
	\end{equation}
	Equation \eqref{contradiction} says that
	\begin{equation}\label{contradiction-2}
	d(x_k)^{2s-\beta} \MM^+_\KK( d^{\beta}(x_k))= \int_{\R^n} \frac{S^+(\delta(d^{\beta},x_k,y))}{d_k^{\beta-2s} |y|^{n+2s}} dy \leq o(1).
	\end{equation}

	Denoting for simplicity $d_k:=d(x_k)$, and performing the change of variables $y= d_k z$, we can rewrite the  integral in \eqref{contradiction-2} as
	\begin{equation} \label{integ}
			\int_{\R^n} \frac{  S^+ (g_k(z))  }{|z|^{n+2s}}  dz,
	\end{equation}
	where
	$$
	g_k(z):=\left(\frac{d(x_k+d_k z)}{d_k}\right)^{\beta}+
			\left(\frac{d(x_k-d_k z)}{d_k}\right)^{\beta}-2.
	$$
	By taking $k \to \infty$, from \eqref{contradiction-2} and \eqref{contra-final2} together with Lemma \ref{lema.c} we arrive at $c^+(\beta)\leq 0$ for $\beta\in(\beta_1,2s)$, which contradicts Lemma \ref{lema.beta}.

	The proofs of (b), (c) and (d) are analogous.
\end{proof}

\begin{rem}
The use of Lemma \ref{lema.beta} is not indispensable in order to obtain the contradiction in the proof of Lemma \ref{lema.m} neither the existence of $\beta_1$ and $\beta_2$. In fact, the same thesis can be obtained by studying the strict concavity of the real-valued function
$$
	C(\tau)=\int_{\R} \frac{(1+t)^\tau_+ + (1-t)^\tau_+ -2}{|t|^{1+2s}}dt,
$$
which is well-defined for $\tau\in(0,2s)$.

\end{rem}

The following lemma is key in order to obtain the boundary regularity for \eqref{main.eq.1}.

\begin{lema} \label{lema.comp}
Let $u$ be a solution of \eqref{main.eq.1} with $s\in (\frac{1}{2},1)$, then there exists $\delta>0$ and $\beta\in (0,\beta_1)$ such that 
$$
	|u(x)|\leq C d(x)^\beta \quad  \forall x \in \Omega_\delta
$$
for some positive constant $C$.
\end{lema}

\begin{proof}
	First, we claim that there exist $\delta>0$, $\beta\in (0,\beta_1)$ and a positive constant $C$ such that
	\begin{align} \label{ecuu1}
		I\xi(x)\leq -Cd(x)^{\beta-2s} \qquad \mbox{in }\Omega_\delta.	
	\end{align}
	provided that $s>1/2$.
	
	We apply Lemma \ref{lema.m}. For $\delta>0$ small enough,  given $x\in \Omega_\delta$ it holds that
	\begin{equation*}
		\mathcal{M}^+_\KK \xi \leq -Cd(x)^{\beta-2s} \quad \mbox{ in }\Omega_\delta.
	\end{equation*}	  
	for some $C>0$ and $\beta \in (0,\beta_1)$.
	Now, since $\nabla\xi(x)=C\beta d(x)^{\beta-1}$, we have
	\begin{align*}
		I \xi(x) &\leq \mathcal{M_\LL^+} \xi(x) = \mathcal{M}_\KK^+ \xi(x) +  c^+ |\nabla \xi(x)|\\
		&\leq -C d(x)^{\beta-2s} + C d(x)^{\beta-1}\\
		& \leq -C d(x)^{\beta-2s}
	\end{align*}
	whenever $\beta-2\alpha < \beta-1$, that is, $s>\frac{1}{2}$, from where claim \eqref{ecuu1} follows.

	Moreover, $\delta$ can be taken small enough such that 
	$$
		-I\xi(x) \geq  f(x) \quad \mbox{ in } \Omega_\delta.
	$$
	
	Since, for some positive constant $L$, $\xi(x)= \ell(x) \geq L$ in $\Omega \setminus \Omega_\delta$, we can take $C$ such that $C\xi(x)\geq CL\geq \|u\|_{L^\infty(\Omega)}$ for $x\in \Omega\setminus \Omega_\delta$.  By using that $u$ and $\xi$ vanish in $\Omega^c$ we conclude that
	$$
		C\xi(x)\geq  u(x) \quad \mbox{ in } \Omega_\delta^c.
	$$
	From the comparison principle given in Lemma \ref{lema.comp} we obtain that
	$$
		Cd(x)^\beta = C\xi(x)\geq  u(x) \quad \mbox{ in } \Omega_\delta
	$$
	and the result follows.

Repeating the same argument with $-u$ we find the result.

\end{proof}

\begin{lema}[Boundary regularity] \label{regu.b}
Let $u$ be a solution of \eqref{main.eq.1} with $s\in (s,\frac{1}{2})$. Then there exist $\delta>0$ and $\beta\in (0,\beta_1)$ such that
$$
	|u(x)-u(y)|\leq C |x-y|^\beta \quad \mbox{ for } y\in \partial \Omega, x\in \Omega_\delta.
$$
\end{lema}
\begin{proof}
	Let $y\in \partial \Omega$ and $x\in \Omega_\delta$. Since $u(y)=0$, by using Lemma \ref{lema.comp} we get that
	\begin{equation} \label{x1}
		 |u(x)|=|u(y)-u(x)| \leq C d(x)^\beta.
	\end{equation}
	Gathering \eqref{x1} and the definition of $d(x)$ we obtain	
	\begin{equation*}
	  d(x)^\beta=\inf_{\tilde x\in \partial \Omega} |x-\tilde x|^\beta \leq   |x-y|^\beta
	\end{equation*}	
	and the proof is complete.
\end{proof}

In \cite{CL}, by applying a diminish of oscillation argument (see, for instance \cite{CS,CS2}), the following interior H\"older regularity for \eqref{main.eq.1} is proved.

\begin{lema}[Interior regularity, \cite{CL}] \label{regu.i}
	Let $f\in L^\infty$ and $u$ be a viscosity solution of 
	\begin{equation} \label{eqb1}
	\begin{cases}
		-Iu=f &\qquad \mbox{ in }B_1\\
		u=0  &\qquad \mbox{ in }\R^n\setminus B_1.
	\end{cases}	
	\end{equation}
	Then $u\in C^\alpha(B_{1/2})$ for some universal $\alpha\in (0,1)$, and satisfies,
	\begin{equation} \label{cota1}
		\|u\|_{C^\alpha(B_{1/2})} \leq C(\|u\|_{L^\infty(B_1)}+\|f\|_{L^\infty(B_1)})
	\end{equation}
	for some universal $C$.
\end{lema}

Finally, combining the interior and boundary regularity given in Theorem \ref{regu.i} and Theorem \ref{regu.b}, by an standard ball covering argument it follows the next result.

\begin{thm}[Global regularity] \label{t.reg.g}
	Let $\Omega\subset\R^n$ a bounded domain, $f\in L^\infty(\Omega)$ and $u$ be a viscosity solution of \eqref{main.eq} with $s\in (\frac{1}{2},1)$.
	Then $u\in C^\gamma(\bar \Omega)$ for $\gamma=\min\{\alpha,\beta\}$, where $\alpha$ and $\beta$ are given in Theorems \ref{regu.i} and \ref{regu.b}, respectively.
\end{thm}
 
\section{Aleksandrov-Bakelman-Pucci estimate}

The Aleksandrov-Bakelman-Pucci (ABP) estimate is  a key ingredient in our arguments. It is the relation that allows us to pass from an estimate in measure to a pointwise 
estimate. In this section, we prove an ABP estimate for integro-differential equations with  gradient term by following the argument in \cite{CS2}. 
In \cite{CL}, Chang-Lara also given a version of ABP estimate involving the operator $I$, see Theorem 3.4 in \cite{CL}. However, he used the ABP estimate to prove the regularity and we can not use it directly to prove the maximum principle in narrow domain. Therefore, in this section we prove an ABP estimate following the ideas in \cite{CS2} and also \cite{KRS}.

Let $u$ be a function that is not positive outside the ball $B_1$. Consider its concave envelope $\Gamma$ in $B_3$ defined as
\begin{equation} \nonumber
	\Gamma(x):=\begin{cases}
		\min\{p(x): \mbox{for all planes} \,\,p>u \,\,\mbox{in} \,\,B_3\} &\qquad \mbox{ in } B_3\\
	  0  &\qquad \mbox{ in }\R^n\setminus B_3.
	\end{cases}	
	\end{equation}
We define in this way the (non empty) set of sub differentials of $\Gamma$ at $x$, denoted by $\nabla\Gamma(x)$, which will coincide with its gradient,
and also the gradient of $u$, when these functions are differential.

\begin{lema}\label{l4.1}
Let $u\le0$ in $\R^n\setminus B_1$. Let $\Gamma$ be its concave envelope in $B_3$. Assume that
\[
 \mathcal M_{\mathcal K}^+ u(x)+ c^+ |\nabla u(x)|\ge-f(x)\quad in\,\,B_1
\]
with positive constant $c^+$. Given $\rho_0>0$, we define $r_k=\rho_02^{-\frac{1}{2(1-s)}-k}$ and $R_k(x)=B_{r_k}(x)\setminus B_{r_{k+1}}(x)$.

Then, there is  a constant $C_0$ depending on $n,\lambda,$ but not on $s$ such that for any $x\in\{u=\Gamma\}$ and any $M>0$, there is  a $k$ such that
\[
|R_k(x)\cap\{u(y)<u(x)+(y-x)\cdot\nabla \Gamma(x)-Mr_k^2\}\le C_0\frac{F(x)}{M}|R_k(x)|,
\]
where
\[
F(x)=\left((c^+)^n+\frac{f(x)^n}{\mu^n}\right)^{1/n}\left(|\nabla\Gamma(x)|^{\frac{n}{n-1}}+\mu^{\frac{n}{n-1}}\right)^{(n-1)/n}
\]
for some positive constant $\mu$.
\end{lema}
\begin{proof}
We follow the proof of Lemma 8.1 in \cite{CS2}. We just need the following estimate, by some $\mu>0$ and Young's ineuality,
\begin{align*}
 \mathcal M_{\mathcal K}^+ u(x)&\ge -(f(x)+c^+|\nabla u(x)|)\\
 &=-\left(\mu\cdot\frac{f(x)}{\mu}+c^+\cdot|\nabla u(x)|\right)\\
 &\ge-\left((c^+)^n+\frac{f(x)^n}{\mu^n}\right)^{1/n}\left(|\nabla\Gamma(x)|^{\frac{n}{n-1}}+\mu^{\frac{n}{n-1}}\right)^{(n-1)/n}\\
 &=-F(x).
\end{align*}
By choosing $C_0=\rho^{2(s-1)}C$ large enough  and   a similar argument as Lemma 8.1 in \cite{CS2}, we get our estimate.  
\end{proof}

\begin{lema}\label{l4.2} (Lemma 8.4 in \cite{CS2})
Let $\Gamma$ be a concave function in $B_r$. Assume that for a small $\ve$
\[
|\{y:\Gamma(y)<\Gamma(x)+(y-x)\cdot \nabla\Gamma(x)-h\}\cap(B_r\setminus B_{r/2})|\le \ve|B_r\setminus B_{r/2}|,
\]
then $\Gamma(y)\ge \Gamma(x)+(y-x)\cdot\nabla \Gamma(x)-h$ in the whole ball $B_{r/2}$.
\end{lema}

A direct conclusion of Lemmas \ref{l4.1} and \ref{l4.2} is the following

\begin{cor}\label{c4.1}
For any $\ve_0>0$, there is $r\in(0,\rho_02^{-\frac{1}{2(1-s)}})$ such that  function $u$ with the same hypothesis as in Lemma \ref{l4.1} satisfying
\begin{equation}\label{e}
\frac{|\{y\in B_r\setminus B_{r/2}(x):u(y)<u(x)+(y-x)\cdot \nabla\Gamma(x)-C_0F(x)r^2/\ve_0\}|}{|B_r(x)\setminus B_{r/2}(x)|}\le \ve_0.
\end{equation}
\begin{equation}\label{ee}
|\nabla\Gamma(B_{r/4}(x))|\le (8C_0/\ve_0)^nF(x)^n|B_{r/4}(x)|.
\end{equation}
\end{cor}

\begin{proof}
From Lemma \ref{l4.1}, we have (\ref{e}) by choosing $M=C_0F(x)/\ve_0$. 

Next, we prove inequality (\ref{ee}). First note that for every $b>0$ the set $\{y\in\R^n: \Gamma(y)<\Gamma(x)+(y-x)\cdot(\nabla\Gamma-b)\}$ is a subset of
$\{y\in\R^n: u(y)<u(x)+(y-x)\cdot(\nabla\Gamma-b)\}$. Using this relation and (\ref{e}) we conclude that there is a constant $C\ge1$ and some $r\in (0,\rho_02^{-\frac{1}{2(1-s)}})$  such that
\begin{equation}\label{eee}
\frac{|\{y\in B_r\setminus B_{r/2}(x):\Gamma(y)<\Gamma(x)+(y-x)\cdot \nabla\Gamma(x)-C_0F(x)r^2/\ve_0\}|}{|B_r(x)\setminus B_{r/2}(x)|}\le \ve_0.
\end{equation}
Because of the concavity of $\Gamma$ and (\ref{eee}), we may apply  lemma \ref{l4.2} for $h=C_0F(x)r^2/\ve_0$. We obtain that
\[
\Gamma(y)\ge \Gamma(x)+(y-x)\cdot\nabla \Gamma(x)-C_0F(x)r^2/\ve_0
\]
for every $x\in B_{r/2}(x)$. At he same time,
\[
\Gamma(y)\le \Gamma(x)+(y-x)\cdot\nabla\Gamma(x)
\]
for every $y\in B_{r/2}(x)$ because of the concavity of $\Gamma$. Hence,
\[
|\Gamma(y)-\Gamma(x)-(y-x)\cdot\nabla\Gamma(x)|\le C_0F(x)r^2/\ve_0
\]
for every $y\in B_{r/2}(x)$. Since $F$ is a positive function, by Lemma 4.5 (ii) in \cite{KRS} , we have that
\[
|\nabla\Gamma(B_{r/4}(x))|\le (8C_0/\ve_0)^nF(x)^n|B_{r/4}(x)|.
\]
This completes the proof.
\end{proof}

Now, we can prove the following ABP estimate by using Corollary \ref{c4.1}. 

\begin{thm}\label{t4.1}(ABP estimate) Let $f$ is a continuous function and bounded by above and $s\in (1/2,1)$.  Suppose $\sup_{B_1}u<\infty$ and $u$ is a viscosity solution of
\[
 \mathcal M_{\mathcal K}^+ u(x)+ c^+ |\nabla u(x)|\ge-f(x)\quad in\,\,B_1,
\]
$u\le 0$ in $\R^n\setminus B_1$. Then
\[
\sup_{B_1}u^+\le C\|f\|_{L^\infty(B_1)}|B_1|^{1/n},
\]
where positive constant $C$ depending on $n, \rho_0,C_0, c^+,\ve_0$ but on on $s$. 
\end{thm}

\begin{proof} Here we follow some arguments in \cite{GT}.
Recalling $F(x)$ in Lemma \ref{l4.1} and (\ref{ee}), we have
\begin{equation}\label{1}
\frac{|\nabla\Gamma(B_{r/4}(x))|}{|\nabla\Gamma(x)|^n+\mu^n}\le \left(\frac{8C_0}{\ve_0}\right)^n\left((c^+)^n+\frac{f(x)^n}{\mu^n}\right)|B_{r/4}(x)|.
\end{equation}
On the other hand, by (\ref{e}) and Lemma \ref{l4.2}, we get
\[
\Gamma(y)\ge \Gamma(x)-(y-x)\cdot \nabla\Gamma(x)-\frac{C_0}{\ve_0}F(x)r^2, \quad y\in B_{r/2}(x).
\]
Moreover
\[
\Gamma(y)\le \Gamma(x)+(y-x)\cdot\nabla\Gamma(x)
\]
for every $y\in B_{r/2}(x)$ because of the concavity of $\Gamma$. Therefore, we get
\[
|p-\nabla\Gamma(x)|\le \frac{C_0}{\ve_0}F(x)r, \quad p\in\nabla\Gamma(B_{r/4}(x)).
\]
Hence, for $p\in\nabla\Gamma(B_{r/4}(x))$,
\[
|\nabla\Gamma(x)|\le |p|+ \frac{C_0}{\ve_0}F(x)r
\]
which implies
\begin{align*}
|\nabla\Gamma(x)|^n+\mu^n&\le C\left[|p|^n+\mu^n+\left(\frac{C_0r}{\ve_0}\right)^nF(x)^n\right]\\
&\le C\left[|p|^n+\mu^n+\left(\frac{C_0r}{\ve_0}\right)^n\left((c^+)^n+\frac{\|f\|_{L^\infty(B_1)}^n}{\mu^n}\right)\left(|\nabla\Gamma(x)|^n+\mu^n\right)\right].
\end{align*}
Notice that $C_0=C\rho_0^{2s-2}$ and $r\le \rho_02^{-\frac{1}{2-2s}}$, then $rC_0\le C\rho_0^{2s-1}$. Choosing $\rho_0$ small enough, we get that
\begin{equation}\label{2}
|\nabla\Gamma(x)|^n+\mu^n\le C[|p|^n+\mu^n]
\end{equation}
for all $p\in\nabla\Gamma(B_{r/4}(x))$. 

Consequently, from (\ref{1}) and (\ref{2}), we obtain that
\[
\int_{\nabla\Gamma(B_{r/4}(x))}\frac{dp}{p^n+\mu^n}\le \left(\frac{8C_0}{\ve_0}\right)^n\left((c^+)^n+\frac{f(x)^n}{\mu^n}\right)|B_{r/4}(x)|.
\]
Set $M=\sup_{B_1}u^+$. If $u^+ =0$ in $\R^n\setminus B_1$ and $u$ is upper semicontinuous,
there is $x_0\in B_1$ with $M=u^+(x_0)$. Next, we consider  a covering on $B_1$ by balls $B_{r_i/4}(x_i)=B_i$  and $0<r_i<4$ for $i=0,1,2,\cdots, m\in \mathbb{N}$, then we obtain that
\[
\int_{\nabla\Gamma(B_1)}\frac{dp}{p^n+\mu^n}\le \left(\frac{8C_0}{\ve_0}\right)^n\left((c^+)^n+\frac{f(x)^n}{\mu^n}\right)|B_i|.
\]
Since $B_{\frac{M}{4}}\subset \nabla\Gamma(B_1)$ (see Lemmas 9.2 and 9.4 in \cite{GT}), then we have
\[
\log\left(\left(\frac{M}{\mu}\right)^n+1\right)\le \left(\frac{8C_0}{\ve_0}\right)^n\left((c^+)^n+\frac{\|f(x)\|_{L^\infty(B_i)}^n}{\mu^n}\right)|B_i|.
\]
If $f^+\not\equiv0$, let $\mu=\|f(x)\|_{L^\infty(B_i)}|B_i|^{1/n}$, we have
\begin{align*}
\sup_{B_1}u^+&\le \left(\exp\left\{ \left(\frac{8C_0}{\ve_0}\right)^n\left((c^+)^n+1\right)\right\}-1\right)^{1/n}\|f\|_{L^\infty(B_i)}|B_i|^{1/n}\\
&\le C\|f\|_{L^\infty(B_1)}|B_1|^{1/n},
\end{align*}
where constant $C\ge1$.

 If $f\equiv 0$, we choosing $\mu>0$, by a similar argument as above and letting $\mu\rightarrow 0$. This completes the proof.
\end{proof}

\section{(H) condition}
In this section we prove a technical lemma in order to apply the Krein-Rutmann theorem, see the Appendix for details. In this context, to prove the (H) condition is equivalent the analyze the the existence of a bounded non-negative function $f$ and a corresponding viscosity solution of the equation \eqref{main.eq.1} such that $u\leq Kf$ in $\Omega$ for some positive constant $K$.

First, we prove that the function $\xi$ defined in \eqref{f.xi} satisfy the following properties with  $\beta_2$ is the value defined in Lemma \ref{lema.beta}.

\begin{lema} \label{lema.aux}
	Given $\beta\in (\beta_2,2s)$ and $s\in(\frac{1}{2},1)$, the function $\xi(x)$ satisfies
		$$
			\begin{array}{lll}
			(a)  & \mathcal{M}_\LL^-(\xi) \geq C & \mbox{ in } \Omega_\delta,\\
			(b) & \xi(x)=0 & \mbox{ in } \Omega^c,\\
			(c) & \xi(x)\leq  L & \mbox{ in } \Omega\setminus \Omega_\delta.
			\end{array}			
		$$
		where $C$ and $L$ are positive constants depending on $\delta$.
\end{lema}

\begin{proof}
Lemma \ref{lema.m} ensures that $\xi(x)$ satisfies the inequality
$$
	\MM^-_\KK \xi(x)\geq C d(x)^{\beta-2s} \quad \mbox{ in } \Omega_\delta
$$
provided that $\beta\in(\beta_2,2s)$, for $\delta>0$ small enough. Moreover,  since $\xi=0$ in $\Omega^c$ we have that $D\xi(x)= \beta d(x)^{\beta-1}$. We get
\begin{align} \label{in.1}
\begin{split}
	\MM^-_\LL(\xi(x))&= \MM^-_\KK(\xi(x))-c^+ |D\xi(x)|\\
	&\geq Cd(x)^{\beta-2s}-\beta d(x)^{\beta-1}\\
	&\geq C d(x)^{\beta-2s} 
\end{split}	
\end{align}
in $\Omega_\delta$ provided that $s>1/2$.
\end{proof}

\medskip

\begin{lema}[(H) condition] \label{lema.cond}
	There exist a non-negative function $f$ and a positive constant $K$  such that
	$$
		u\geq Kf \quad \mbox{in } \Omega,
	$$
	where $u$ is a viscosity solution of \eqref{main.eq.1}. 
\end{lema}

\begin{proof}
	Given $f\geq 0$, let $u$ be a nontrivial viscosity solution of \eqref{main.eq.1}, that is 
	$$
		-I(u)=f \mbox{ in } \Omega,\qquad u=0 \mbox{ in }\Omega^c.
	$$
	By using the Strong maximum principle stated in Theorem \ref{teo.pm}  it follows that $u>0$ in $\Omega$. We define 
	$$
	w:=K\xi, \quad \mbox{ where	} \quad K:= L^{-1}\inf_{x\in \Omega\setminus \Omega_\delta }u(x).
	$$
	where $\xi$ is the function defined in \eqref{f.xi} and $L=L(\delta)$ is the constant given in Lemma \ref{lema.aux}.	From property (a) of Lemma \ref{lema.aux} it follows that
	$$
		I(w)\geq \mathcal{M}_\LL^-(w)= K\mathcal{M}_\LL^-(\xi)\geq KC \geq 0 \quad \mbox{ in } \Omega_\delta
	$$
	which implies that $w$ is a sub-solution of $-I(u)=f$ in $\Omega_\delta$ since
	$$
		-I(w)\leq 0 \leq f=-I(u)  \quad \mbox{ in } \Omega_\delta.
	$$
	Property (b) of Lemma \ref{lema.aux} leads to $w=0\leq u$ in $\Omega^c$.	Moreover, property (c) gives that
	$$
		w\leq K L \leq u \mbox{ in } \Omega\setminus \Omega_\delta.
	$$
	Therefore, from the Comparison principle given in Lemma \ref{lema.comp} it follows that $w\leq u$ in $\Omega$, and hence we finally obtain that $u\geq K\xi$ in $\Omega$.

\end{proof}

\section{Proof of main results}
This section is devoted to prove Theorems \ref{teo.main} and \ref{teo.main 2}. We start this section by a maximum principle in small domains.
\begin{thm}\label{t1}
Let $f$ is a continuous function and bounded by above. There exists $\varepsilon_0>0$, depending on $n,\lambda,\Lambda,s,c^+$ and $|\Omega|$, such that if $|\Omega|\leq \varepsilon_0$ then for any $u\in LSC(\Omega)\cap L^1(\omega_s)$ and bounded by above,
\begin{equation*} 
\begin{cases}
	\mathcal{M}^-_{\LL}u\le f &\qquad \mbox{ in }\Omega,\\
	u\ge 0  &\qquad \mbox{ in }\R^n\setminus \Omega
\end{cases}	
\end{equation*}
implies $u\geq0$ in $\Omega$.
\end{thm}

\begin{proof}
 Let $f=\sup_{\Omega} u^-$,
then by the ABP estimate (we just need to  extend the ABP estimate in unit ball which is proved in Theorem \ref{t4.1} to a general domain $\Omega$) we have that
\[
\sup_{\Omega}u^-\le C |\Omega|^{1/n} \sup_{\Omega} u^-\le \frac{1}{2}\sup_{\Omega} u^-
\]
if $|\Omega|$ small. Hence we get $u^-\equiv 0$ in $\Omega$ which means $u\ge0$ in $\Omega$. 
\end{proof}

The following  theorem is needed.
\begin{thm}\label{t2}
Suppose $u,v \in  C(\bar{\Omega})\cap L^1(\omega_s)$ and $f\in  C(\Omega)$ satisfy
\begin{equation*} 
\begin{cases}
	Iu\le f &\qquad \mbox{ in }\Omega,\\
	u>0  &\qquad \mbox{ in }\Omega,\\
	u\ge0 &\qquad \mbox{ in }\R^n\setminus\Omega,
\end{cases}	
\quad resp.
\begin{cases}
	Iu\ge f &\qquad \mbox{ in }\Omega,\\
	u<0  &\qquad \mbox{ in }\Omega,\\
	u\le0 &\qquad \mbox{ in }\R^n\setminus\Omega
\end{cases}
\end{equation*}
and 
\begin{equation*} 
\begin{cases}
	Iv\ge f &\qquad \mbox{ in }\Omega,\\
	v\le0  &\qquad \mbox{ in }\R^n\setminus\Omega,\\
	v(x_0)>u(x_0),
\end{cases}	
\quad resp.
\begin{cases}
	Iv\le f &\qquad \mbox{ in }\Omega,\\
	v\ge0  &\qquad \mbox{ in }\R^n\setminus\Omega,\\
	v(x_0)<u(x_0),
\end{cases}
\end{equation*}
for some point $x_0\in\Omega$ and $f\le0$ (resp. $f\ge0$). Then $u\equiv tv$ for some $t>0$.
\end{thm}

\begin{proof}
Let $u,v$ satisfy the first set of inequalities in Theorem \ref{t2}. Take a compact set $K\subset\Omega$ such that $|\Omega\setminus K|\le \varepsilon_0$, where $\varepsilon_0$
is given in Theorem \ref{t1}. Set $z_t=v-tu$. If $t$ is large enough $z_t<0$ in $K$.  For $t\ge 1$, we have
\[
\mathcal{M}^+_\LL z_t\ge Iv-tIu=(1-t)f\ge0\quad{\rm in}\,\,\Omega
\]
and $z_t\le0$ in $\R^n\setminus(\Omega\setminus K)$, by using Theorem \ref{t1} we get $z_t\le0$ in $\Omega\setminus K$ and thus $z_t\le0$ in $\Omega$.
So, by the strong maximum principle, either $z_t\equiv0$ in $\Omega$ in which case we are done, or $z_t<0$ in $\Omega$. We define
\[
\tau=\inf\{t\,\,|\,\, z_t<0\,\,{\rm in}\,\,\Omega\}.
\]
Since $v(x_0)>u(x_0)$ we have $\tau>1$. Now we repeat the same argument for $z_\tau$. So, either $z_\tau\equiv0$ in $\Omega$ in which case we are done or $z_\tau<0$ in $\Omega$.
In this case there exists $\eta>0$ such that $z_{\tau-\eta}<0$ in $K$. Now we repeat again the same argument for $z_{\tau-\eta}$, which yields a contradiction with the definition of $\tau$.

If the inequalities satisfies by $u,v$ are reversed (second set of inequalities in Theorem \ref{t2}), we consider the function $tu-v$ and the same argument.
\end{proof}

\begin{rem}\label{r}
Here we remark that if $f\equiv 0$ in Theorem \ref{t2}, we just need $v(x_0)>0$ instead of $v(x_0)>u(x_0)$ (resp. $v(x_0)<0$ instead of $v(x_0)<u(x_0)$).
See also  Theorem 4.2 in \cite{QS} for local case.
\end{rem}

A consequence of Theorem \ref{t2} is an upper bound the of the principal half-eigenvalue in terms of thickness of the domain. For each $\rho\in\R$, we define a nonlinear operator $G_\rho$
by 
\[
G_\rho(u)=-Iu-\rho u.
\]
We say the operator $G_\rho$ satisfies the maximum principle in $\Omega$ if , whenever $v\in LSC(\Omega)\cap L^1(\omega_s)$ is a solution of $G_\rho v\le 0$ in $\Omega$ with $v\le 0$ in $\R^n\setminus \Omega$, we have $v\le0$ in $\Omega$; Similarly,  We say the operator $G_\rho$ satisfies the minimum principle in $\Omega$ if , whenever $v\in LSC(\Omega)\cap L^1(\omega_s)$ is a solution of $G_\rho v\ge 0$ in $\Omega$ with $v\ge 0$ in $\R^n\setminus \Omega$, we have $v\ge0$ in $\Omega$.

Define constants
\[
\mu^+(I,\Omega)=\sup\{\rho: G_\rho \mbox{ satisfies the maximum principle in }\Omega\},
\]
and
\[
\mu^-(I,\Omega)=\sup\{\rho: G_\rho \mbox{ satisfies the minimum principle in }\Omega\},
\]
We will eventually show $\lambda_1^{\pm}(I,\Omega)=\mu^{\pm}(I,\Omega)$. The following  lemma is the first step in this direction.

\begin{lema}\label{p1}
We have
\[
\lambda_1^{\pm}(I,\Omega)\le\mu^{\pm}(I,\Omega)<\infty.
\]
\end{lema}

\begin{proof} Here we follow the argument as Lemma 3.7 in \cite{A}.
We will show
\[
\lambda_1^{+}(I,\Omega)\le\mu^{+}(I,\Omega).
\]
Suppose on the contrary $\mu^{+}(I,\Omega)<\rho_1<\rho_2<\lambda_1^{+}(I,\Omega)$. Then we may select a function $v_1$ which satisfies
\[
-Iv_1\le \rho_1v_1 \quad {\rm in}\,\, \Omega
\]
and such that $v_1\le 0$ in $\R^n\setminus \Omega$ and $v_1>0$ somewhere in $\Omega$.  We can also select $v_2$ such that $v_2>0$ in $\Omega$, $v_2\ge0$ in  $\R^n\setminus \Omega$ and $v_2$ satisfies
\[
-Iv_2\ge \rho_2v_2 \quad {\rm in}\,\, \Omega.
\]
Since $\rho_1v_2<\rho_2 v_2$, we may apply Theorem \ref{t2} to deduce $v_2=tv_1$ for some $t>0$.
This implies $\rho_1=\rho_2$, a contradiction. Hence,   $\lambda_1^{+}(I,\Omega)\le\mu^{+}(I,\Omega)$. By a similar argument, we can obtain $\lambda_1^{-}(I,\Omega)\le\mu^{-}(I,\Omega)$.

Finally, we will prove the operator $G_\rho$ does not satisfy the minimum principle in $\Omega$ for all large $\rho$. Choosing a continuous function
$h\le0$, $h\not\equiv0$ with compact support in $\Omega$. By Theorem \ref{t2.6}, there exists a unique solution of the following problem
\begin{align*}
	\begin{cases}
		-Iv=h &\quad \mbox{ in } \Omega,\\
		v=0&\quad\mbox{ in } \R^n\setminus\Omega.
	\end{cases}
	\end{align*}
According to the comparison principle, $v\le 0$ in $\Omega$. Since $h\not\equiv0$, we have $v\not\equiv0$. Hence, $v<0$ in $\Omega$ by the strong maximum principle.
Since $h$ has compact support in $\Omega$ we may select a constant $\rho_0>0$ such that $\rho_0v\le h$. Therefore, $v$ satisfying
\[
-Iv\ge \rho_0v \quad{\rm in}\,\,\Omega
\]
an so evidently the operator $G_\rho$ does not satisfy the minimum principle in $\Omega$, for any $\rho\ge \rho_0$. Thus
$\lambda^{-}(I,\Omega)\le\rho_0$. By a similar argument, we have that $\lambda^{+}(I,\Omega)<\infty$.
\end{proof}

Next, we  prove Theorem \ref{teo.main 2} by using Theorem \ref{t2}.

\begin{proof}[\bf
Proof of Theorem \ref{teo.main 2}]

We shall use Theorem \ref{t2} (the first set of inequalities), with $Iu$ replaced by $Iu+\lambda^+_1(I,\Omega) u$, $f\equiv0$ and Remark \ref{r}. 

Suppose $u_1=u$ satisfies (\ref{1.3}). Then we apply Theorem \ref{t2} with $u=\phi^+_1$ and $v=u_1$. 

If $u_1=u$ satisfies (\ref{1.2}), then either $u_1$ is positive somewhere, so $u_1$ satisfies (\ref{1.3}) and we are in the previous case,
or $u_1$ is a negative eigenfunction. Then $\lambda_1^+=\lambda_1^-$, by Theorem \ref{teo.main}. Then, we apply  Theorem \ref{t2} with $u=\phi^+_1$ and $v=-u_1$. This completes the proof.
\end{proof}

The proof of  Theorem \ref{teo.main} follows by using the Krein-Rutman Theorem. In order to give the proof we introduce some notation and definitions. We set the space 
$$
	X:=\{f\in C(\R^n) \,:\, f=0 \mbox { in } \R^n\setminus\Omega\},
$$
and we denote  $K$ the closed convex cone in $X$ with vertex $0$
$$
	K:=\{f\in X \, : \, f\geq 0 \mbox{ in } \Omega\}.
$$
The cone $K$ induces an ordering $\preceq$ on $X$ as follows: given $f,g \in X$ we say that
	$$
	f\preceq g \iff g-f \in K.
	$$
Given $f\in L^\infty(\R^n)$, let $u$ be a viscosity solution of 
\begin{equation} \label{main.eq.2}
\begin{cases}
	-Iu=f &\qquad \mbox{ in }\Omega\\
	u=0  &\qquad \mbox{ in }\R^n\setminus \Omega.
\end{cases}	
\end{equation}
 Since $I$ is invertible, we define the solution operator $T$ as
$$
	T(f):=I^{-1}(-f)=u.
$$

\begin{proof}[\bf Proof of Theorem \ref{teo.main}]
	We check that the hypothesis of  Theorem \ref{teo.kr} are fulfilled. 

	The operator $T$ is \emph{positively $1-$homogeneous}. 
	Given $t>0$, we have that $T(tf)=u$ where $u$ is a viscosity solution of $-I(u)=tf$ in $\Omega$, $u=0$ in $\Omega^c$. Since $I$ is a $1-$homogeneous operator it holds that $f=-I(t^{-1} u)$, from where follows that $tT(f)=u$.
	
	From Proposition \ref{prop.cont} it follows that $T$ is a \emph{continuous operator} on $X$. Moreover, by using the Holder regularity up the boundary of $I$  given in Theorem \ref{t.reg.g} and the Arzela-Ascoli theorem, it follows that $T$ is a \emph{compact operator} on $X$.
		
	The order \emph{$\preceq$ is increasing}.
	Given $f,g\in X$ such that $f\preceq g$, let $u$ and $v$ be viscosity solutions of $-Iu=f$, $-Iv=g$ in $\Omega$ and $u=v=0$ in $\Omega^c$.	By definition of the order, we get that $-I(u)=f\le g=-I(v)$ in $\Omega$, and $u=v=0$ in $\Omega^c$. Hence, by using the Comparison principle given in Lemma 2.4, it follows that $u\leq v$ in $\R^n$, from where $T(f)\preceq T(g)$.

	 Moreover, the order \emph{$\preceq$ is strictly increasing}. If now $f\neq g$ are functions such that  $f\prec g$, by definition of the order, and by using Theorem \ref{teo.sol}, we obtain that
	$$
		-\mathcal{M}^-_\LL(v-u) \geq g-f>0 \quad \mbox{ in } \Omega.
	$$
	Applying the Maximum Principle stated in Theorem \ref{teo.pm} it follows that $v-u>0$ in  $\Omega$, from where $T(f)\prec T(g)$.
	
	Finally, the \emph{(H) Condition} in this context means that there exists a non-zero function $f\in K$ and a positive constant $M$ such that $f \preceq MT(f)$. This conditions is equivalent to analyze the existence of  functions $u$ and $f$ such that for some positive constant $M$ it holds that $f\leq Mu$ in $\Omega$,
	where $u$ is a viscosity solution of $-I(u)=f$ in $\Omega$, $u=0$ in $\Omega^c$ and $f\geq 0$. Such affirmation is proved in  Lemma \ref{lema.cond}.
	
	Consequently, there exists a positive eigenfunction $f\in K$ of $T$ which is unique up to a multiplicative constant, and $\mu$, the corresponding eigenvalue is simple and it can be characterized as the eigenvalue having the smallest absolute value. Observe that for $\mu\neq 0$, we have $T(f)=\mu f$ if and only if $-I(f)=\lam^* f$ for $\lam^*=\frac{1}{\mu}$. 
	
It is now immediate from the definitions of $\mu^+(I,\Omega)$ and $\lambda^+(I,\Omega)$ that $\mu^+(I,\Omega)\le \lambda^*\le \lambda^+(I,\Omega)$, and therefore $\lambda^*=\mu^+(I,\Omega)= \lambda^+(I,\Omega)$ by Lemma \ref{p1}. By a similar argument, we know $\lambda^-(I,\Omega)$ is also the eigenvalue of operator $-I$. We complete the proof.

\end{proof}

\section{An application: Decay estimates for the evolution equation}
	In this section we are interested in the asymptotic behavior as $t\to\infty$ of the solutions of a evolution-type equation involving the operator $I$ defined in \eqref{opp}. In order to state our results, it is convenient to define the notion of viscosity solution in this context.

	We denote the cylinder of radius $r$, height $\tau$ and center $(x,t)\in \R^n\times \R$ by $C_{r,\tau}(x,t):=B_r(x)\times(t-\tau,\tau)$. 

	In this context, we define the space of lower and upper semicontinuous functions as follows. 

	\begin{defn}
		$LSC((t_1,t_2]\to L^1(\omega_s))$ consists of  all  measurable  functions $u:\R^n\times (t_1,t_2] \to \R$ such that for every $t\in(t_1,t_2]$,
		\begin{itemize}
			 \item[i)]$\|u(\cdot,t)^-\|_{L^1(\omega_s)}<\infty$,
			 \item[ii)] $\lim_{\tau\to 0} \|(u(\cdot,t)-u(\cdot,t-\tau))^+\|_{L^1(\omega_s)}=0$.
		\end{itemize}
		 Similarly, $u\in USC((t_1,t_2]\to L^1(\omega_s))$ if $-u\in LSC((t_1,t_2]\to L^1(\omega_s))$. We finally denote $C((t_1,t_2]\to L^1(\omega_s))=LSC((t_1,t_2]\to L^1(\omega_s))\cap USC((t_1,t_2]\to L^1(\omega_s))$.
	\end{defn}

	A lower semicontinuous test function is a pair $(\varphi,C_{r,\tau}(x,t))$ such that 
	$$
		\varphi\in C_x^{1,1}C_t^1(C_{r,\tau}(x,t))\cap LSC((t-\tau,\tau]\to L^1(\omega_s)).
	$$ 
	Similarly, $(\varphi,C_{r,\tau}(x,t))$ is an upper semicontinuous test function if the pair $(-\varphi,C_{r,\tau}(x,t))$ is a lower semicontinuous test function.

	\begin{defn}
		Given an elliptic operator $I$, a function $u\in LSC(\Omega\times(t_1,t_2]) \cap LSC((t_1\times t_2]\to L^1(\omega_s))$ is said to be a \emph{viscosity super solution} to $u_t\geq Iu$ in $\Omega\times (t_1,t_2]$, if for every lower semicontinuos test function $(\varphi,C_{r,\tau}(x,t))$ and   $(x,t)\in\Omega\times (t_1,t_2]$, whatever
		\begin{itemize}
			\item[i)]  $\varphi(x,t)=u(x,t)$ and
			\item[ii)]$\varphi(y,s)\leq u(y,s)$ for $(y,s)\in\R^n\times (t-\tau,t]$,
		\end{itemize}
		 we have that $\varphi_{t}(x,t)\geq I\varphi(x,t)$.
	\end{defn}

	The definition of $u$ being a \emph{viscosity sub solution} to $u_t\leq Iu$ in $\Omega\times (t_1\times t_2]$ is done similarly to the definition of super solution replacing $LSC$ by $USC$ and reversing the last two inequalities. Finally, a \emph{viscosity solution} to $u_t=Iu$ in $\Omega\times (t_1,t_2]$ is a function which is a super and sub solution  simultaneously. 
	
	\bigskip

	Let $u$ be a viscosity solution to the parabolic equation
	\begin{align} \label{calor}
	\begin{cases}
		u_t=Iu &\quad \mbox{ in } \Omega\times (0,\infty),\\
		u(x,0)=h_0(x) &\quad\mbox{ in } \Omega\times \{0\},\\
		u(x,t)=0&\quad\mbox{ in } \partial\Omega\times  (0,\infty).
	\end{cases}
	\end{align}
	We are interested in the asymptotic behavior, as $t\to\infty$, of the solution $h(x,t)$ of \eqref{calor}. Based on results of the local heat equation, one expects $h$ to decay to zero exponentially and that the rate of decay and the extinction profile are somehow connected with the principal eigenvalue $\lam$ and the eigenfunction $v$ given in Theorem \ref{teo.main}, i.e., 
	\begin{equation}  \label{ecu.e}
	\begin{cases}
		-Iv=\lam v &\qquad \mbox{ in }\Omega\\
		v=0  &\qquad \mbox{ in }\R^n\setminus \Omega.
	\end{cases}	
	\end{equation} 
	In contrast with the ordinary heat equation, since an orthonormal basis of eigenfunctions for the space is not present, precise estimates are much harder to obtain.
	Due to the lack of a condition replacing the orthogonality in this settings, instead of obtaining estimates for the difference $|h(x,t)e^{\lam_1 t}-v_1(x)|$, we are only able to estimate the logarithmic difference 
	$$
		\log(h(x,t)e^{\lam t})-\log v(x)=\log\left(\frac{h(x,t)e^{\lam t}}{v(x)}\right).
	$$

	\begin{prop}
	Let $h$, $v$ and $\lam$ be as above. We have that
	$$
		\sup_{\Omega\times (0,\infty)}  \frac{h(x,t)}{v(x)e^{-\lam t}} 
		\leq \sup_{\Omega} \frac{h_0^+(x)}{v(x)},	
	$$
	where $h_0^+=\max\{h_0,0\}$ denotes the positive part of $h_0$.
	\end{prop}

	\begin{proof}
		 By replacing $h_0$ with its positive part if necessary, we may assume that the initial data h 0 is non-negative.
		It clearly suffices to show that
		$$
			\sup_{\Omega\times (0,T)}  \frac{h(x,t)}{v(x)e^{-\lam t}} 
			= \max\left\{\sup_{\Omega} \frac{h_0^+(x)}{v(x)},0 \right\}
		$$
		for any $T>0$. We argue by contradiction and suppose that
		$$
			0< \frac{h(x_0,t_0)}{v(x_0)e^{-\lam t_0}}  = \sup_{\Omega\times (0,T)}  \frac{h(x,t)}{v(x)e^{-\lam t}} 
		$$	
		for some $x_0\in \Omega$ and $0<t_0\leq T$. We denote $Q$ a neighborhood of $(x_0,t_0)$ where $h$ is positive.
		We define the function $w_\ve(x)=e^{-\lam t} v(x) +  \frac{\ve}{T-t}$. An straightforward computation shows that 
		\begin{equation}  \label{xxxxx}
			(w_\ve)_t> Iw_\ve.
		\end{equation}
			
		Moreover, $w_\ve(x,t)\to\infty$ uniformly in $x$ as $t\to T$ and the function $h-w_\ve$ has a local maximum in $Q$ for $\ve>0$ small enough. For simplicity of notation, we denote this maximum point also $(x_0,t_0)$ and notice that $t_0<T$.
		
		Since $h$ is a viscosity solution of \eqref{calor}, the last claim implies that	$(w_\ve)_t\leq I(w_\ve)$, with contradict inequality \eqref{xxxxx}, and the proof follows.
	 \end{proof}
	 
	 \begin{cor}
		 Let $h$ be a viscosity solution of \eqref{calor} with $h_0\in C(\bar \Omega)$. Then
		 $$
			 \sup_\Omega |h(x,t)|=o(e^{-\lam t}) \qquad \mbox{for all } \lam<\lam_1(\Omega)
		 $$
		 being $\lam_1$ the principal eigenvalue of \eqref{ecu.e}.
	 \end{cor}

\section{Appendix: the Krein--Rutman Theorem }
	Let $X$ be a real Banach space. Let $K$ be a closed convex cone in $X$ with vertex $0$, i.e., 
	\begin{itemize}
	\item $0\in K$
	\item $x\in K, t\in \R^+$ then $tx\in K$
	\item $x,y \in K$ then $x+y \in K$
	\end{itemize}
	We further assume that
	$$ K\cap-K =\{0\}.$$
	The cone $K$ induces an ordering $\preceq$ on $X$ as follows. Given $x,y \in X$ we say that
	$$
	x\preceq y \iff y-x \in K.
	$$
	A mapping $T:X \to X$ is said to be increasing if $x\preceq y \Rightarrow T(x) \preceq T(y)$. The mapping is said to be compact if it takes bounded subsets of $X$ into relatively compact subsets of $X$. We say that the mapping is positively $1-$homogeneous if it satisfies the relation $T(tx)=tT(x)$ for all $x\in X$ and $t\in \R^+$.
	
	\begin{thm}[Krein-Rutmann for non-linear operator, \cite{M}] \label{teo.kr}
		Let $T:X\to X$ be an increasing, positively $1-$homogeneous compact continuous operator(non-linear) on $X$ for which there exists a non-zero $u\in K$ and $M>0$ such that
		\begin{equation*}
			\bold{(H)}\qquad u \preceq M T u.
		\end{equation*}
		Then, $T$ has a non-zero eigenvector $x_0\in K$. Furthermore, if $K$ has non-empty interior and if $T$ maps $K\setminus\{0\}$ into $K^\circ$ and is strictly increasing, then $x_0$ is the unique positive eigenvector in $K$ up to a multiplicative 	constant. And, finally if $\mu_0$ be the corresponding eigenvalue, then it can be characterized as the eigenvalue having the smallest absolute value and furthermore, it is simple.
	\end{thm}

 \setcounter{equation}{0}
\section{ Acknowledgements}
The authors would like to express their thanks to P. Felmer and B. Sirakov for their valuable comments on the ABP estimate.
A. Quaas was partially supported by Fondecyt Grant No. 1151180 Programa Basal, CMM. U. de Chile and Millennium Nucleus Center for Analysis of PDE NC130017.


\begin{thebibliography}{10}	

\bibitem{A}
S. Armstrong, \emph{Principal eigenvalues and an anti-maximum principle for homogeneous fully nonlinear elliptic equations}. Journal of Differential Equations, 246(7), 2958-2987, 2009.
\bibitem{A1}
S. Armstrong, \emph{The Dirichlet problem for the Bellman equation at resonance.} J. Differential Equations, 247 (2009), pp. 931-955.

\bibitem{BCI}        
G. Barles, E. Chasseigne and C. Imbert. \emph{On the dirichlet problem for second-order elliptic integro-differential equations}. Indiana Univ. Math. J., 57(1):213-246, 2008.  

\bibitem{B}
H. Berestycki, \emph{On some nonlinear Sturm-Liouville problems}. J. Differential Equations 26 (3) (1977) 375-390.



\bibitem{BDGQ}
B. Barrios, L. Del Pezzo, J. Garc\'ia-Meli\'an and A. Quaas,
\emph{A priori bounds and existence of solutions for some nonlocal elliptic problems}. arXiv preprint arXiv:1506.04289.

\bibitem{BD}
I. Birindelli and F. Demengel, \emph{First eigenvalue and maximum principle for fully nonlinear singular operators}, Adv. Differential Equations 11 (1) (2006) 91-119.

\bibitem{BD2}
I. Birindelli and F. Demengel, \emph{Eigenvalue, maximum principle and regularity for fully nonlinear homogeneous operators}, Commun. Pure Appl. Anal. 6 (2) (2007) 335-366.




\bibitem{BEQ}  
J. Busca, M. Esteban and A. Quaas, \emph{Nonlinear eigenvalues and bifurcation problems for Pucci's operator}, Ann. Inst. H. Poincare Anal. Non Lineaire 22 (2) (2005) 187-206.

\bibitem{BNV}    
H. Berestycki, L. Nirenberg  and S. Varadhan, \emph{The principal eigenvalue and maximum principle for second order elliptic operators in general domains}, Comm. Pure Appl. Math. 47 (1) (1994) 47-92.


\bibitem{CC}    
L. Caffarelli and X. Cabr\'e, \emph{Fully nonlinear elliptic equations}, volume 43. American Mathematical Society Colloquium Publications, 1995.

\bibitem{CS}    
  L. Caffarelli and L.  Silvestre, \emph{Regularity results for nonlocal equations by approximation}. Archive for rational mechanics and analysis, 200(1), 59-88.
  
\bibitem{CS2}      
  L. Caffarelli and L. Silvestre, \emph{Regularity theory for fully nonlinear integro-differential equations}. Comm. Pure Appl. Math., 62(5):597-€"638, 2009.



\bibitem{CL}
  H.A. Chang-Lara,  \emph{Regularity for fully non linear equations with non local drift}. arXiv preprint arXiv:1210.4242.


\bibitem{FQ}
P. Felmer and A. Quaas, \emph{Boundary blow up solutions for fractional elliptic equations}. Asymptotic Analysis, 78(3), 123-144.

\bibitem{FQ2}
P. Felmer and A. Quaas,  \emph{Positive solutions to "semilinear" equation involving the Pucci's operator}, J. Differential Equations 199 (2) (2004) 376-393. 

\bibitem{FQS1}
P. Felmer, A. Quaas  and B. Sirakov, \emph{Resonance phenomena for second-order stochastic control equations. } SIAM J. Math. Anal. 42 (2010), no. 3, 997-1024. 

\bibitem{FQS2}
P. Felmer, A. Quaas  and B. Sirakov, \emph{Landesman-Lazer type results for second order Hamilton-Jacobi-Bellman equations. } J. Funct. Anal. 258 (2010), no. 12, 4154-4182.




\bibitem{GT}
D. Gilbarg and S. Trudinger,  \emph{Elliptic Partial Differential Equations of Second Order}, Classics Math., Springer-Verlag, Berlin, 2001, reprint of the 1998 edition.

\bibitem{H}      
H.M.  Soner, \emph{Optimal control with state-space constraint}. II. SIAM J. Control Optim.,
24(6):1110-€"1122, 1986.

\bibitem{IY}  
H. Ishii and Y. Yoshimura, \emph{Demi-eigenvalues for uniformly elliptic Isaacs operators}, preprint.

\bibitem{KR}    
 M.G. Krein and M.A. Rutman, \emph{Linear operators leaving invariant a cone in a Banach space}, Amer. Math. Soc. Transl. 10
(1962) 199-€"325.

\bibitem{KRS}    
M. Kassmann, M. Rang and R. Schwab, \emph{Integro-differential equations with nonlinear directional dependence}. Indiana Univ. Math. J. 63 (2014), no. 5, 1467-1498.

\bibitem{KS}    
N.V.  Krylov and M.V.  Safonov. \emph{An estimate on the probability that a diffusion process hits a set of positive measure}. Doklady Akademii Nauk SSSR, 245(1):18-€"20,
1979.  

  
  
  \bibitem{L}
P.-L. Lions, \emph{Bifurcation and optimal stochastic control}, Nonlinear Anal. 7 (2) (1983) 177-207.

  \bibitem{M}
R. Mahadevan, \emph{A note on a non-linear Krein-Rutman theorem.} 
Nonlinear Anal. 67 (2007), no. 11, 3084-3090. 

  \bibitem{P}
C. Pucci, \emph{Maximum and minimum first eigenvalues for a class of elliptic operators}, Proc. Amer. Math. Soc. 17 (1966) 788-795.

  \bibitem{QS}
A. Quaas and B. Sirakov,  \emph{Principal eigenvalues and the Dirichlet problem for fully nonlinear elliptic operators}, Adv. Math. 218 (1) (2008) 105-135. 

  \bibitem{QS2}
A. Quaas and B. Sirakov,  \emph{On the principle eigenvalues and the Dirichlet problem for fully nonlinear operators}, C. R. Acad. Sci. Paris 342 (2006) 115-118.


  \bibitem{R}
P.H. Rabinowitz, \emph{Some global results for nonlinear eigenvalue problems.} J. Funct. Anal. 7 (1971) 487-513.

\bibitem{ROS}  
 X. Ros-Oton and J. Serra, \emph{Boundary regularity for fully nonlinear integro-differential equations}. Duke Math. J., to appear.  arXiv preprint arXiv:1404.1197.
  

\bibitem{S}    
L. Silvestre, \emph{H\"older estimates for solutions of integro-differential equations like the
fractional laplace}. Indiana Univ. Math. J., 55(3):1155-€"1174, 2006.

\bibitem{S2}    
B. Sirakov, \emph{Non-uniqueness for the Dirichlet problem for fully nonlinear elliptic operators and the Ambrosetti-Prodi phenomenon}, Analysis and topology in nonlinear differential equations, 405-421, Progr. Nonlinear Differential Equations Appl., 85, Birkh\"auser/Springer, Cham, 2014.



\end{thebibliography}
\end{document}